\documentclass[a4paper, 12pt]{article}
\usepackage{mathtext}
\usepackage{amsmath}
\usepackage{amsfonts}
\usepackage{amssymb}
\usepackage{mathrsfs}
\usepackage{amsthm}
\usepackage[colorlinks=true, linkcolor=blue,citecolor=red, unicode=true, 
	pageanchor=false, bookmarks = false, pdfpagemode=UseNone, pdfstartview=FitH]{hyperref}

\theoremstyle{definition}\newtheorem{Def}{Definition}
\theoremstyle{plain}\newtheorem{Th}{Theorem}
\theoremstyle{remark}\newtheorem{Rem}{Fact}
\theoremstyle{plain}\newtheorem{Le}{Lemma}
\theoremstyle{plain}\newtheorem{Cor}{Corollary}

\newcommand{\Cam}{\dot{C}}

\newcommand{\BMO}{\mathrm{BMO}}
\newcommand{\Lip}{\mathrm{Lip}}

\DeclareMathOperator{\supp}{supp}

\author{Nikolay N. Osipov\thanks{The author is supported by Chebyshev Laboratory of SPbU (RF Government grant no. 11.G34.31.0026), 
by RFBR (grant no. 11-01-00526), and by Rokhlin grant.}}
\title{\vskip-1cm Littlewood--Paley--Rubio de Francia inequality in Morrey--Campanato spaces}

\begin{document}
\maketitle

{\small
\begin{center}
\vskip-1cm
St.~Petersburg Department 
of Steklov Mathematical Institute RAS,\\Fontanka 27, St.~Petersburg, 
Russia
\end{center}
\begin{center}
Chebyshev Laboratory (SPbU), 14th Line 29B, Vasilyevsky Island, St.~Petersburg, Russia
\end{center}

\begin{center}
{\em 
nicknick@pdmi.ras.ru
} 
\end{center}}
\vskip1cm

\begin{abstract}
Rubio de Francia proved the one-sided Littlewood--Paley inequality for arbitrary intervals in $L^p$, $2\le p<\infty$.
In this article, his methods are developed and employed to prove an analogue of such an inequality 
``beyond the index $p=\infty$'', i.e., for spaces of H\"older functions and $\BMO$.
\end{abstract}

\section{History}\label{Hist}
Let $\{\Delta_m\}$ be a finite or countable collection of mutually disjoint intervals in~$\mathbb{R}$.
By $M_{\Delta_m}$ we denote the operators corresponding to the multipliers~$\chi_{\Delta_m}$: $M_{\Delta_m} f = (\widehat{f}\chi_{\Delta_m})^{\vee}$.
In 1983, Rubio de Francia (see~\cite{Ru}) proved that
\begin{equation}\label{RFEstimate}
	\Big\|\Big(\sum_m |M_{\Delta_m}f|^2 \Big)^{1/2}\Big\|_{L^p(\mathbb{R})} \le C_p\|f\|_{L^p(\mathbb{R})},
	\quad 2 \le p < \infty,
\end{equation}
where the constant $C_p$ does not depend on~$f$ or~$\{\Delta_m\}$.
By duality, this estimate is equivalent to the following one:
\begin{equation}\label{DualEstimate}
	\Big\|\sum_{m}f_{m}\Big\|_{L^p(\mathbb{R})} \le
	C_p\Big\|\Big(\sum_m |f_m|^2 \Big)^{1/2}\Big\|_{L^p(\mathbb{R})},\quad 1<p \le 2,
\end{equation}
where $\{f_m\}$ is a sequence of functions such that $\supp{\widehat{f}_{m}} \subset \Delta_{m}$.
In 1984, Bourgain (see~\cite{Bo}) proved that estimate~\eqref{DualEstimate} remains true for $p=1$.
His method was more complicated then Rubio de Francia's arguments. But in 2005, Kislyakov and Parilov employed 
a technique similar to Rubio de Francia's and established (see~\cite{KiPa}) 
that estimate~\eqref{DualEstimate} is fulfilled for all $0 < p \le 2$. 
For $p\le 1$, this is sooner an $H^p$- than an $L^p$-estimate, because the main techniques lean heavily upon the properties of Calder\'on--Zygmund operators on ``real variable'' Hardy classes~$H^p$ (see, e.g.,~\cite{FeSt, St}).

Now, the dual of $H^1$ is $\BMO$ and the dual of $H^p$ for $p<1$ is a certain H\"older class, see~\cite{Du,St}. It is natural to ask whether there exists a dual counterpart of~\eqref{DualEstimate} for $0<p\le 1$.
Note that, unlike the case of ${1<p\le 2}$, naive dualization is impossible because the $H^p$-classes arise in the proof of~\eqref{DualEstimate} in a somewhat tricky way. By way of explanation, we observe that a function~$f$ with
spectrum in an interval~$[a,b]$ gives rise to at least two natural $H^p$-functions, namely, $e^{-2\pi i\, ax} f(x)$ and $e^{-2\pi i\, bx} f(x)$ (one of them is ``analytic'' and the other ``antianalytic''), and the two are required
in the proof of~\eqref{DualEstimate}.

To be more specific, we signalize that in \cite{Ki, KiPa, Ru}, the authors studied, in fact, operators of the following form:
$$
	S^{1}f(x)=\Big\{ e^{-2\pi i\,a_m x}M_{\Delta_m}f(x)\Big\}\quad\mbox{and}\quad
	S^{2}f(x)=\Big\{ e^{-2\pi i\,b_m x}M_{\Delta_m}f(x)\Big\}\rule{0pt}{20pt}
$$
(or their adjoints), where $a_m$ and $b_m$ are the left and the right end of~$\Delta_m$ respectively. It is well known and easily seen that these operators are not
bounded on the spaces dual to $H^p$. We show this for the space $\BMO = (H^1)^\ast$. Suppose 
$$
	\|M_{[0,\eta]}(f)\|_\BMO \le C\|f\|_{\BMO},
$$ 
where $f\in L^2\cap\BMO$ and $C$ does not depend on $f$ or $\eta$.
Consider some function $\varphi \in C^\infty(\mathbb{R})$ such that $\varphi \equiv 1$ on $[1,2]$ and $\varphi \equiv 0$ outside $[0,3]$.
Then the operators corresponding to the Fourier multipliers~$\varphi(\eta\xi)$, $\eta>0$, are uniformly bounded from $L^1$ to $H^1$ (because they map $L^1$ to functions with spectrum in~$\mathbb{R}_+$) and, therefore,
from $\BMO$ to~$L^{\infty}$. But this means that $\|f\|_{\BMO}\asymp\|f\|_{L^{\infty}}$, provided $\supp \widehat f \subset [\eta,2\eta]$.
On the other hand, if we put $\Delta = \big[0,\tfrac{3}{2}\eta\big]$, then $\|M_\Delta f\|_{\BMO}\le C\|f\|_{\BMO}$ by our assumtion. 
If $\supp \widehat f \subset [\eta,2\eta]$, this implies that $\|M_\Delta f\|_{L^\infty}\le C\|f\|_{L^\infty}$. 
But therefore, the Riesz projection (the antianalytic one) is uniformly bounded in~$L^{\infty}$ on all the functions such that their Fourier
transforms have compact supports. This is a contradiction.

However, after a slight modification, the operators~$S^1$ and~$S^2$ become bounded on $\BMO$ and spaces of H\"older functions. It is only necessary to smooth out 
each multiplier~$\chi_{\Delta_m}$ at one of the ends of $\Delta_m$. 
Here is quite a particular case of what will be proved. Consider some functions $\psi^1$ and $\psi^2$ in $C^{\infty}([0,1])$ such that $\psi^1$ supported in $\big[0, \tfrac{2}{3}\big]$, $\psi^2$ supported in $\big[\tfrac{1}{3}, 1\big]$, 
and $\psi^1 + \psi^2 \equiv 1$ on $[0,1]$. We extend these functions to the whole line by zero. After that, using shifts and dilations, we can obtain 
functions~$\psi^1_m$ and~$\psi^2_m$ such that $\psi^1_m + \psi^2_m\equiv \chi_{\Delta_m}$. Each such function is smooth at one of the ends of the corresponding interval. Now we redefine the operators~$S^1$ and~$S^2$, using the
functions~$\psi^1_m$ instead of $\chi_{\Delta_m}$ for $S^1$ and the functions~$\psi^2_m$ for $S^2$. It turns out that if $f \in L^2(\mathbb{R}) \cap \BMO(\mathbb{R})$, then $S^\sigma f \in \BMO(\mathbb{R})$, $\sigma = 1,2$, and 
$\|S^\sigma f\|_{\BMO}\le C\|f\|_{\BMO}$. The same can be said if we take, instead of $\BMO$, a H\"older class $\Lip_s(\mathbb{R})$, $0<s<1$ (see~\eqref{LipDef}).
We will also see that it is possible to tell something substantial in the spirit of formula~\eqref{RFEstimate} for functions that satisfy the $\BMO$ condition or a H\"older condition at only one point.

It is also worth mentioning that multidimensional analogs of~\eqref{RFEstimate} and~\eqref{DualEstimate} exist for all $p$, $0<p \le \infty$, (see \cite{Jo, Os1, Os2}) and they are not a direct consequence of the one-dimensional results. 
It would be interesting to obtain a multidimensional version of the results presented in this paper.

\section{Preliminaries}\label{Prelim}
In our context, it is most natural to endow spaces of smooth functions with Morrey--Campanato norms.

\paragraph{Morrey--Campanato spaces.}
Let $\mathcal{P}_i$ be the space of algebraic polynomials of degree strictly less than~$i$. We agree that ${\mathcal{P}_0 = \{0\}}$.
For $l^2$-valued polynomials, we use the notation $\mathcal{P}_i(l^2)$. 
Now we give the definition of the Morrey--Campanato spaces~$\Cam_{p}^{s,i}$.
\begin{Def}
	Suppose $i\in\mathbb{Z}_+$, $1\le p<\infty$, and $s\in(-n/p,i]$.
	Let $f$ be a locally integrable function on~$\mathbb{R}^n$ (scalar-valued or $l^2$-valued). 
	We say that $f \in \Cam_{p}^{s,i}$ if
	$$
		\|f\|_{i,p,s} = \sup_{Q}\inf_{P}\frac{1}{|Q|^{s/n}}\bigg(\frac{1}{|Q|}\int\limits_{Q}\big|f-P\big|^p\bigg)^{1/p}<\infty,
	$$
	where the supremum is taken over all the cubes in $\mathbb{R}^n$ and the infimum is taken over all the polynomials in 
	$\mathcal{P}_i$ or $\mathcal{P}_i(l^2)$.
\end{Def}
The John--Nirenberg inequality implies that $\BMO = \Cam_{p}^{0,1}$ for all ${p\in [1,\infty)}$ (in the sense that the norms are
equivalent). We also state a counterpart of this for $s>0$.
\begin{Le}
	If $i\in\mathbb{N}$ and $0< s\le i$\textup, then $\Cam_{p}^{s,i}$ does not depend on~$p$\textup, ${1\le p<\infty}$.
\end{Le}
The proof can be found in~\cite{Ca,Me}; see also the exposition in Section~1.1.2 of~\cite{KiKr}. 
For $s>0$, this equivalence of norms extends naturally to $p=\infty$ in a sense. More precisely, in this case Morrey--Campanato spaces coincide with standard spaces of functions with a power-type modulus of smoothness. 
Again, see~\cite{Ca,Me} and the exposition in~\cite{KiKr}. First, we formulate separately the result for $0<s\le 1$.
\begin{Le}
	If $0<s\le 1$ and $1\le p<\infty$\textup, then the space $\Cam_{p}^{s,1}$ coincides with the Lipschitz class $\Lip_s$ that is determined by 
	the seminorm
	\begin{equation}\label{LipDef}
		\|f\|_{\Lip_s} = \sup_{x\ne y}\frac{|f(x)-f(y)|}{|x-y|^s}.
	\end{equation}
\end{Le}

In fact, all the Campanato spaces (the spaces~$\Cam_{p}^{s,i}$ for $s>0$) can be renormed in a similar way.
Namely, for a function~$f$ on~$\mathbb{R}^n$ and $h\in\mathbb{R}^n$, define the differences $\Delta_h^\alpha f$ by the recurrent formulas
$$
	\Delta_h^1 f(x) = f(x+h)-f(x),\quad \Delta_h^\alpha f(x) = \Delta_h^1\Delta_h^{\alpha-1}f(x),\quad \alpha=1,2,\dots.
$$

\begin{Le}
Let $i\in\mathbb{N}$\textup, $1\le p<\infty$\textup, and $0<s\le i$.
Then for the space~$\Cam_{p}^{s,i}$\textup, we can choose $\alpha \in \mathbb{N}$ such that $\Cam_{p}^{s,i}$ coincides with the class of 
locally integrable functions
satisfying
$$
|\Delta_h^{\alpha} f(x)|\le C|h|^s,
$$
where $C$ does not depend on~$x$ or~$h$. The infimum of such~$C$ provides an equivalent seminorm on~$\Cam_{p}^{s,i}$.
\end{Le}
For example, the spaces~$\Cam_{p}^{1,2}$ coincide with the Zygmund class~$Z$ that is determined by the condition 
$|\Delta_h^{2} f(x)|\le C|h|$.
Details can be found in Section~4.1.1 of~\cite{KiKr}.

Finally, when $s<0$ we obtain the so-called Morrey spaces.

\paragraph{Maximal functions.} Here we introduce the maximal functions corresponding to the Morrey--Campanato norms.
Such maximal operators were studied thoroughly in~\cite{DeSh}.
\begin{Def}\label{DefOfM}
Suppose $s\in\mathbb{R}$, $1\le p< \infty$, and $i\in\mathbb{Z}_+$. 
Let $h = \{h_m\}$ be a collection (finite or countable) of measurable functions on $\mathbb{R}^n$.
We define the maximal function $M_{i,p,s}h$ by the formula
$$
	M_{i,p,s}h(x)=
	\sup_{Q \ni x}\inf_{P}\frac{1}{|Q|^{s/n}}\bigg(\frac{1}{|Q|}\int\limits_{Q}\big|h-P\big|_{l^2}^p\bigg)^{1/p},
$$
where the supremum is taken over all the cubes containing~$x$ and the infimum is taken over all the collections $P = \{P_m\}$
of polynomials $P_m \in \mathcal{P}_i$.
\end{Def}
We have presented this definition in the form that is slightly different from the classic one. Specifically, we do not assume that $h$ belongs to $L^1(l^2)$ and even that~$h$ and~$P$ are $l^2$-valued functions.
The point is that we are going to apply maximal operators in a context in which we do not know whether our sequence~$h$ is an $l^2$-valued function or not. However it turns out that if $M_{i,p,s}h$ is 
finite at least at one point, then $h$ can be ``corrected'' to an $l^2$-valued function. Namely, we can state the following simple fact.
\begin{Rem}\label{AlmostGood}
	Let $h = \{h_m\}$ be a countable collection of measurable functions on $\mathbb{R}^n$ such that 
	$M_{i,p,s}h(x_0)<\infty$ at some point~$x_0$. Then there exists a collection $P = \{P_m\}$ 
	of polynomials $P_m \in \mathcal{P}_i$ such that $h-P \in L_{\mathrm{loc}}^p(\mathbb{R}^n,l^2)$.\footnote{Once again, we emphasize that neither $h$ nor $P$ is assumed to be an $l^2$-valued function.}
\end{Rem}
\begin{proof} 
Consider some cube~$Q_1$ containing~$x_0$.
By Definition~\ref{DefOfM}, we can choose a collection~$P_1$ of polynomials such that $h-P_1 \in L^p(Q_1,l^2)$.
Also, consider another cube $Q_2$, $Q_2 \supset Q_1$, and a corresponding collection $P_2$ such that $h-P_2 \in L^p(Q_2,l^2)$. 
We have 
$$
	P_2-P_1 = (h-P_1)-(h-P_2)\in L^p(Q_1,l^2).
$$
Thus, $P_2-P_1\in \mathcal{P}_i(l^2)$ and 
$$
	h-P_1 = (h-P_2)+(P_2-P_1) \in L^p(Q_2,l^2).
$$ 
Therefore, we may set $P=P_1$.
\end{proof}
In its turn, for a ``true'' function (for example, it can be a sequence corrected through the procedure described above) we can state the following proposition.
\begin{Le}\label{fIsSuit}
	Suppose $i\in\mathbb{Z}_+$\textup, $1\le p<\infty$\textup, and $s\in(-n/p,i]$.
	Let $\beta$ be a positive number such that $\beta > \max \{s,i-1\}$.
	Consider a measurable function~$f$ \textup(scalar-valued or $l^2$-valued\textup) on $\mathbb{R}^n$ such that $M_{i,p,s}f$ is finite at 
	least at one point. Then
	the function 
	$$
		|f(x)|(1+|x|^{n+\beta})^{-1}
	$$ 
	is integrable.
\end{Le}
This lemma was implicitly proved in Section~4.4.1 of~\cite{KiKr}.

Also we need a certain modification of $M_{i,p,s}f$.
Quite often it is convenient to introduce certain projections~$\mathfrak{P}_Q$ that act from $L^p(Q; \tfrac{dx}{|Q|})$ onto $\mathcal{P}_i$ and are bounded in $L^p(Q; \tfrac{dx}{|Q|})$ uniformly in~$Q$.
Such projections
allow us to choose polynomials at which the infimums in the definition of~$M_{i,p,s}f$ are attained roughly.
Namely, for a scalar-valued or $l^2$-valued function~$f\in L^1_{\mathrm{loc}}(\mathbb{R}^n)$, the function $M_{i,p,s}f$ is equivalent 
to the following expression:
\begin{equation}\label{DefOfM2}
	\widetilde M_{i,p,s}f(x) = \sup_Q\frac{1}{|Q|^{s/n}}\bigg(\frac{1}{|Q|}\int\limits_{Q}|f-\mathfrak{P}_Qf|^p\bigg)^{1/p},
\end{equation}
where the supremum is taken over all the cubes centered at $x$.
Details can be found in~\cite{DeSh} and also in Sections~1.2.2 and~4.4.1 of~\cite{KiKr}.

\section{Main result and its discussion}
First, we state a preliminary version of the main result. Namely, in order to avoid complications with the definition of Fourier multipliers, we additionally assume $f\in L^2(\mathbb{R})$.
\begin{Th}\label{MainT}
	Let $\Delta_m=[a_m,b_m]$ be pairwise disjoint intervals on $\mathbb{R}$ with 
	lengths $l_m = b_m - a_m$. Here the index~$m$ runs over some finite or countable set~$\mathcal{M}$.
	Suppose $0 \notin (a_m,b_m)$ for any $m\in\mathcal{M}$.
 	Consider two functions $\psi^1$ and $\psi^2$ in $C^{\infty}([0,1])$ such that 
 	\begin{gather*}
 		\supp{\psi^1}\subset \big[0, \tfrac{2}{3}\big],\quad \supp{\psi^2}\subset \big[\tfrac{1}{3}, 1\big], \\
 		\psi^1 + \psi^2 \equiv 1\quad \mbox{on}\quad [0,1].\rule{0pt}{20pt}
 	\end{gather*}
 	We extend the functions~$\psi^{\sigma}$\textup, $\sigma = 1,2$\textup, to the whole line by zero and introduce
 	their transformations supported on ${\Delta_m}$\textup: 
	$$
		\psi^{\sigma}_m(\xi) = \psi^{\sigma}\bigg(\frac{\xi-a_m}{l_m}\bigg).
	$$	
	Let $f\in L^2(\mathbb{R})$.
	We define two linear operators~$S^{\sigma}$ by the formulas
	\begin{equation}\label{SDef}
	\begin{aligned}
		&S^{1}f(x)=\Big\{ e^{-2\pi i\,a_m x}\big(\widehat f\psi^{1}_m\big)^{\vee}(x)\Big\}_{m\in\mathcal{M}}, \\
		&S^{2}f(x)=\Big\{ e^{-2\pi i\,b_m x}\big(\widehat f\psi^{2}_m\big)^{\vee}(x)\Big\}_{m\in\mathcal{M}}.\rule{0pt}{20pt}
	\end{aligned}
	\end{equation}
	Then we have
	$$
		M_{r,2,s}(S^{\sigma}f) \le C M_{i,2,s}f,
	$$
	provided $i \in \mathbb{Z}_{+}$\textup, $r\in{\mathbb{N}}$\textup, $s\in(-1/2,i]$\textup, 
	and $r>\max \{s,i-1\}$.\footnote{From now on, the letters $s$, $i$, and $r$ always denote parameters for which such relations are
	fulfilled, unless otherwise stated.}
	The constant~$C$ does not depend on $f$ or $\{\Delta_m\}$.	
\end{Th}
By discussion in the previous section, this theorem implies the facts described at the end of Section~\ref{Hist}. However, in general
we do not assume that the function $M_{i,2,s}f$ is finite almost everywhere. In particular, the inequality is true even if this function is finite at only one point, implying a certain
smoothness of~$S^{\sigma}f$ at the same point. A similar remark applies to Theorem~\ref{ExMainT} below, where we lift the restriction $f \in L^2$. Unfortunately, this leads to a more bulky statement.

Suppose $f$ is a function such that $M_{i,p,s}f$ is finite at some point, and $\varphi$ is a function in $\mathcal{S}$. Then
Lemma~\ref{fIsSuit} immediately implies the following: the Fourier transform of~$f$ is well defined, $f\ast\varphi \in C^{\infty}$, 
$f\ast\varphi \in \mathcal{S}'$, and 
$\widehat {f\ast\varphi} = \widehat{f}\,\widehat{\varphi}$. This allows us to state the final version of the main result. 
\begin{Th}\label{ExMainT}
	Let $\Delta_m$ and $\psi^{\sigma}_m$\textup, $\sigma=1,2$\textup, be the intervals and the functions introduced in Theorem~\textup{\ref{MainT}}.
	Suppose $f$ is a measurable function on $\mathbb{R}$ such that $M_{i,2,s}f$ is finite at some point. Then for each $m$\textup, there exist 
	two sequences of functions $\psi^{\sigma}_{m,\nu}\in\mathcal{S}$ \textup(their choice does not depend on~$f$\textup) and two sequences of polynomials 
	$p^{\sigma}_{m,\nu} \in \mathcal{P}_r$ \textup(depending on $f$\textup)\footnote{Here we do not assume that $\{p^{\sigma}_{m,\nu}\}_{m\in\mathcal{M}} \in \mathcal{P}_r(l^2)$.} such that
	\begin{enumerate}
	\item[\textup{1)}] the sequences~$\psi^{\sigma}_{m,\nu}$ converge to the functions~$\psi^{\sigma}_m$ in $L^2(\mathbb{R})$ as $\nu\to \infty$\textup;
	\item[\textup{2)}] 
		there exist two functions $g^\sigma = \{g^{\sigma}_m\}_{m\in\mathcal{M}}$ in $L_{\mathrm{loc}}^2(\mathbb{R},l^2)$ such that
		\begin{align*}
			&\Big\{e^{-2\pi i\,a_m x}\big(\widehat f\psi^1_{m,\nu}\big)^{\vee}(x) - p^1_{m,\nu}(x)\Big\}_{m\in\mathcal{M}}\to g^1, \\ 
			&\Big\{e^{-2\pi i\,b_m x}\big(\widehat f\psi^2_{m,\nu}\big)^{\vee}(x) - p^2_{m,\nu}(x)\Big\}_{m\in\mathcal{M}}\to g^2
		\end{align*}
		as $\nu\to \infty$\textup, where the limits can be taken in $L^2(I,l^2)$ for any interval~$I$\textup{;}
	\item[\textup{3)}] putting 
	$$
		\widetilde S^\sigma f  = g^\sigma,
	$$ 
	we have 
	$$
		M_{r,2,s}\big(\widetilde S^\sigma f\big) \le C M_{i,2,s}f,
	$$ 
	where the constant~$C$ does not depend on $f$ or $\{\Delta_m\}$.
	\end{enumerate}
\end{Th}
It is not hard to see that Theorem~\ref{ExMainT} immediately implies Theorem~\ref{MainT}.

And again, the above local estimates for the maximal functions yield the corresponding norm estimates.
\begin{Cor}
	For $f \in \Cam_{2}^{s,i}(\mathbb{R})$\textup, we have
	$$
		\|\widetilde S^\sigma f\|_{r,2,s} \le C\|f\|_{i,2,s},\quad \sigma=1,2,
	$$
	where the constant~$C$ does not depend on $f$ or $\{\Delta_m\}$.
\end{Cor}
In its turn, this proposition implies that the operators~$\widetilde S^\sigma$ are bounded from $\BMO$ to $\BMO$, 
from $\Lip_s$ to $\Lip_s$ for $0<s<1$, from $\Lip_1$ to $Z$, and from $Z$ to $Z$.

\paragraph{Comparison with Rubio de Francia's considerations.}
It is worth noting that the proof of~\eqref{RFEstimate} in~\cite{Ru} involved a certain $\BMO$-es\-ti\-mate which, however, was weaker than the above results for $s=0$ (but it was sufficient for Rubio de Francia's goals).
To be more specific, we describe some details.
First, Rubio de Francia proved that instead of the operator $Qf = \{M_{\Delta_m}f\}_{m\in\mathcal{M}}$, a more ``regular'' operator~$H$ may be considered (we also encounter such operators in Section~\ref{Conc}, see formula~\eqref{DefOfH}). 
Roughly speaking, he proved that 
$$
	\|Qf\|_{L^p(l^2)}\le C_p\|Hf\|_{L^p(l^2)},\quad 2\le p<\infty.
$$ 
After that, he got the estimate 
\begin{equation}\label{RFIntermEst}
	M_{1,1,0}(|Hf|_{l^2}) \le CM_2f,
\end{equation}
where $M_2f = (M(f^2))^{1/2}$ and $M$ is the usual Hardy--Littlwood maximal function. In fact, he estimated the expression $M_{1,1,0}(Hf)$, but employed the rough estimate~\eqref{RFIntermEst}.
The operator~$M_{1,1,0}$ is exactly the sharp maximal operator $(\cdot)^{\sharp}$, and it is well known that $\|g\|_{L^p}\le C\|g^{\sharp}\|_{L^p}$ (see~\cite{FeSt, St}).
Putting it all together, Rubio de Francia achieved the desired $L^p$-estimate. We can identify two main differences between the considerations just described and our arguments below.
First, we must reduce the operators~$S^{1,2}$ to the ``good'' operator~$H$ in $\BMO$ or in Lipschitz spaces, not in $L^p$-spaces (more precisely, we deal with the corresponding maximal functions).
And second, we prove estimates of the form $M_{r,2,s}(Hf) \le CM_{i,2,s}f$ instead of~\eqref{RFIntermEst}.

\paragraph{Concerning shifts.} 
Note that multiplying each expression $M_{\Delta_m}f$ in~\eqref{RFEstimate} by the factors $e^{-2\pi i\,a_m x}$ or $e^{-2\pi i\,b_m x}$, we can make Rubio de Francia's estimate look more similar to
the claims of Theorems~\ref{MainT} and~\ref{ExMainT}. In fact, those factors arose in the proof of~\eqref{RFEstimate} (in the ``good'' operator~$H$ mentioned above), but they were dropped after taking absolute values. 
On the other hand, we cannot lift them in
Theorems~\ref{MainT} and~\ref{ExMainT}, which deal with smoothness. The theorems say that the resulting $l^2$-valued functions become smooth if their components are shifted (in the frequency range) precisely in this way; it is impossible to shift them back 
because of an uncontrollable oscilation of the above factors. 

However, in the proofs below we do shift smooth functions in the frequency range. Forestalling natural questions, we present a simple example showing how this operation may become innocent (this is not a pattern for what follows, we use heavier technicalities
when we prove the main results). Namely, let $\varphi$ be a Schwartz class function. If $|f(s)-f(t)|\le|s-t|$ for $s,t\in\mathbb{R}$, then for every $a \notin \supp \widehat\varphi$ the function $g = \varphi\ast (e^{2\pi i\,a x} f(x))$ satisfies
$|g(s)-g(t)|\le C|s-t|$ with $C$ independent of~$a$.

To prove this, we recall that $|\varphi'(\tau)|\le \frac{C_N}{1+|\tau|^N}$ for every $N = 1,2,\dots$. Next,
\begin{align*}
	g(s)-g(t) &= \int\limits_{\mathbb{R}}\big[\big(\varphi(s-x)-\varphi(t-x)\big)e^{2\pi i\,a x}\big]f(x) \,dx\\
	&= \int\limits_{\mathbb{R}}\big[\big(\varphi(s-x)-\varphi(t-x)\big)e^{2\pi i\,a x}\big]\big(f(x)-f(t)\big) \,dx,
\end{align*}
because the function in square brackets has zero integral by assumption. Now, we split the integral into the sum of integrals over the set 
$$
	E = \{x\in\mathbb{R}\colon |x-t|\le 2|t-s|\}
$$ 
and its complement. Clearly, the first integral is bounded by
$$
	\int\limits_E \big(|\varphi(s-x)|+|\varphi(t-x)|\big)\,2|t-s|\,dx \le 4|t-s|\int\limits_{\mathbb{R}} |\varphi(\tau)|\,d\tau.
$$
In the second, we estimate the difference $|\varphi(s-x)-\varphi(t-x)|$ by using the Lagrange formula, which yields the bound
$$
	\frac{C|t-s|}{1+|x-t|^N}|x-t|
$$
for the integrand. If $N>2$, then the integral $\int_{\mathbb{R}}\frac{|x-t|\,dx}{1+|x-t|^N}$ converges and does not depend on~$t$, and we are done.

\section{Decomposition of $S^\sigma$}
We prove Theorem~\ref{ExMainT} for the multipliers~$\psi^{1}_{m}$ only. The proof of the other half is symmetric.

We represent the operator in question as a composition of certain auxiliary operators. One of them is an operator of Rubio de Francia type:
instead of cutting~$\widehat f$ into the pieces corresponding to $\Delta_m$, it cuts out smaller pieces that correspond to a more regular partitition. 
Another operator merges these small pieces in such a way that the intervals $\Delta_m$ are formed.
A similar technique was employed in all the previous publications (see~\cite{Ki, KiPa, Ru}), but in their settings it was allowed to add and remove any shifts $e^{-2\pi i\,a x}$ at will, because the $L^p$-norms are shift invariant. We cannot say the same about
the Morrey--Campanato norms or the corresponding maximal functions. So our arguments below are subtler: we have to treat shifts more accurately.

We introduce some objects that are required for our proof.
Let $A$, ${A>1}$, be a number sufficiently close to~$1$. We choose a function~$\theta$ such that 
$\supp\widehat\theta\subset [A^{-1}, A]$ and
$$
	\sum_{v\in\mathbb{Z}} \widehat\theta(\xi/A^v) \equiv 1
$$
on $(0,+\infty)$. 
By $\theta_v$ we denote the functions such that $\widehat\theta_v(\xi) = \widehat\theta(\xi/A^v)$, i.e.
$\theta_v(t)= A^v\theta(A^v t)$, and by $\mathfrak{a}_v$ we denote the intervals $[A^{v-1},A^{v+1}]$. 
Note that $\supp\widehat\theta_v\subset \mathfrak{a}_v$

Using the partition of unity introduced above, we build the sequence~$\psi^1_{m,\nu}$. Recall that $l_m$ is the length of~$\Delta_m$. For each $m$, we consider the intervals~$\mathfrak{a}_v$ that intersect with $\big[0, \frac{2}{3}l_m\big]$ and denote the index of the rightmost interval by $N_m$.
Then we set 
$$
	\psi^1_{m,\nu}(\xi) = \psi^1_m(\xi)\sum_{v=\nu}^{N_m} \widehat\theta_v(\xi - a_m),\quad \nu \le N_m.
$$
It is clear that $\psi^1_{m,\nu} \to \psi^1_m$ in $L^2$ as $\nu\to -\infty$.

Further, we consider the intervals
$$
	J_{k,j} = \big[j2^k,(j+8)2^k\big],\quad k,j\in\mathbb{Z},
$$
which can be obtained from the dyadic ones by 8-fold dilation with preservation of the left ends.
It turns out that any interval in~$\mathbb{R}$ can, in a sense, be approximated by some interval~$J_{k,j}$. Namely, we can prove the following simple fact.
\begin{Rem}
	For any interval $\Delta = [a,b]$, there exist indices $k,j\in\mathbb{Z}$ such that 
	$\Delta\subset \frac{3}{4}J_{k,j}$ and
	$|\Delta| \asymp |J_{k,j}|$.
\end{Rem}
\begin{proof}
	Choose the index $k$ such that 
	$2^k \le |\Delta| < 2^{k+1}$, and set 
	$$
		j = \sup\big\{j \in \mathbb{Z} \colon (j+1)2^{k} < a\big\}.
	$$
	Then it is easy to see that the interval $J_{k,j}$ is required.
\end{proof}

Thus, for each $m$ and $v$, we can choose indices~$k_{m,v}$ and~$j_{m,v}$ such that 
$$
	a_m+\mathfrak{a}_v\subset \tfrac{3}{4}J_{k_{m,v},\,j_{m,v}}
$$
and the lengths of this intervals are comparable. By $a_{m,v}$ we denote the left ends of $J_{k_{m,v},\,j_{m,v}}$, i.e.,
we set $a_{m,v}=j_{m,v}2^{k_{m,v}}$.

Next, for each $k\in\mathbb{Z}$ we build a function $\phi_k$ such that 
$$
	{\supp\widehat\phi_k\subset J_{k,0}}\quad\mbox{and}
	\quad\widehat\phi_k\equiv 1\quad\mbox{on}\quad\tfrac{3}{4}J_{k,0}.
$$
Namely, let $\phi$ be a function in $\mathcal{S}$ such that $\supp\widehat\phi\subset [0,8]$ and $\widehat\phi \equiv 1$ on $[1,7]$.
Then we set $\phi_k(t) = 2^k\phi(2^k t)$.

Also, we need some smooth extension of~$\psi^1$. Namely,
let~$\tilde \psi$ be a function in $\mathcal{S}$ such that $\tilde \psi \equiv \psi^1$ on $[0,1]$ and $\tilde \psi \equiv 0$ on $[1,+\infty)$. Furthermore, we introduce the functions 
$\varphi_m(t) = l_m\tilde\psi^{\vee}(l_m t)$. Note that $\widehat\varphi_m(\xi) = \psi^1_m(\xi+a_m)$ on $[0,l_m]$ for each~$m$.

Finally, we set $\delta_{m,v}=a_{m,v}-a_m$ and
\begin{equation}\label{g}
	g_{m,v}(t)= \int\limits_{\mathbb{R}} \phi_{k_{m,v}}(t-y)e^{-2\pi i\, a_{m,v}y}f(y) \,dy.
\end{equation}
Using all the objects introduced above, we can write the following identity:
\begin{multline*}
	e^{-2\pi i\,a_m x}\big(\widehat f\psi_{m,\nu}\big)^{\vee}(x) \\= 
	\sum_{v=\nu}^{N_m}\int\limits_{\mathbb{R}} e^{2\pi i\,\delta_{m,v}\tau}\theta_v(x - \tau)\int\limits_{\mathbb{R}}
	e^{-2\pi i\,\delta_{m,v}(\tau - t)}\varphi_m(\tau - t)\,g_{m,v}(t)\,d t d \tau,
\end{multline*}
To verify this relation, it suffices to compare
the Fourier transforms of its left and right parts.
Denoting $e^{2\pi i\,\delta_{m,v} \tau}\theta_v(x - \tau)$ by $\rho_{m,v}(x,\tau)$ and 
$e^{-2\pi i\,\delta_{m,v}t}\varphi_m(t)$ by $\Phi_{m,v}(t)$, we have
\begin{equation}\label{Decomp}
	e^{-2\pi i\,a_m x}\big(\widehat f\psi_{m,\nu}\big)^{\vee}(x) = 
	\sum_{v=\nu}^{N_m}\int\limits_{\mathbb{R}} \rho_{m,v}(x,\tau)\, (\Phi_{m,v}\ast g_{m,v})(\tau) \,d\tau.
\end{equation}
As we will see in Section~\ref{Conc}, the sequence $\{g_{m,v}\}$ can be split into a finite number of subsequences in such a way that the intervals $J_{k_{m,v},\,j_{m,v}}$, corresponding to any one of them, are pairwise disjoint. 
Each such subsequence has the form $Hf$, where $H$ is an operator of Rubio de Francia type: it is similar to our initial operator, but the elements of its kernel are smooth and the corresponding intervals have a more regular structure.
The convolutions with $\Phi_{m,v}$ arise for technical reasons: we need to ``tweak'' the spectrum on the right of $\Delta_m$, otherwise identity~\eqref{Decomp} would not be achieved. Finally, we return to the intervals~$\Delta_m$, 
applying the ``merging'' operator with kernel $\{\rho_{m,v}(x,\tau)\}$. All these operators (Rubio de Francia operators, the operator with kernel $\{\Phi_{m,v}\}$, and the merging operator) can be treated as Calder\'on--Zygmund operators.
In order to make this possible for Rubio de Francia operators, we have shifted the left ends $a_{m,v}$ of our small and regular intervals $J_{k_{m,v},\,j_{m,v}}$ into the origin (instead of doing this for~$\Delta_m$). Therefore, in 
the merging operator the shifts $e^{2\pi i\,\delta_{m,v} \tau}$ occur:
we put each small interval $J_{k_{m,v},\,j_{m,v}}$ in its place within the corresponding big interval~$\Delta_m$.

\section{Merging operator~$R$}
First, we study the operator~$R$ that transforms a double sequence of functions to an unary sequence by the formula
$$
	R(\{h_{m,v}\})(x) = 
	\Big\{\sum_{v=-\infty}^{N_m}\int\limits_{\mathbb{R}} \rho_{m,v}(x,\tau)h_{m,v}(\tau)\,d\tau\Big\}_{m\in\mathcal{M}}.
$$
\begin{Rem}\label{RIsL2Bounded}
The operator~$R$ is $L^2$-bounded.
\end{Rem}
\begin{proof}
This follows immediately from the Plancherel theorem and the fact that each set $\supp\widehat\theta_v$ can intersect
at most two other sets $\supp\widehat\theta_{v-1}$ and $\supp\widehat\theta_{v+1}$.
\end{proof}
\begin{Rem}\label{RPIsZero}
For any double sequence~$q$ of polynomials $q_{m,v}$, we have $Rq \equiv 0$.
\end{Rem}
\begin{proof}
The Fourier transform of the function $e^{2\pi i\,\delta_{m,v} \tau} q_{m,v}(\tau)$ is supported at 
the single point $\delta_{m,v}$, which does not belong to $\supp\widehat\theta_v$. Thus,
$$
	\int\limits_{\mathbb{R}} \rho_{m,v}(\cdot,\tau)q_{m,v}(\tau)\,d\tau \equiv 0.
$$
\end{proof}
Also, the kernel of~$R$ satisfies a certain smoothness condition.
\begin{Le}\label{Smooth1}
	Fix some $m\in\mathcal{M}$. For any interval~$I$ 
	with center~$x_0$\textup, 
	there exist an $l^2$-valued function 
	$$
		p_{m,I}(x,\tau) = \big\{p_{m,v,I}(x,\tau)\big\}_{v\in\mathbb{Z}}
	$$ 
	with the following properties.
	First\textup,
	$$
		p_{m,v,I}(x,\tau) = \sum_{\alpha = 0}^{r-1} x^{\alpha}\,b_{\alpha}^{m,v,I}(\tau),
	$$
	where $b_{\alpha}^{m,v,I}$ are functions in $\mathcal{S}$.
	Second\textup,
	\begin{equation}\label{Smooth1Estimate}
		\Big(\sum_{v\in\mathbb{Z}} \big|\rho_{m,v}(x,\tau)-p_{m,v,I}(x,\tau)\big|^2\Big)^{1/2}\le 
		\frac{C_r\, |I|^{r}}{|\tau-x_0|^{r+1}}
	\end{equation}
	for $\tau\notin 2I$ and $x \in I$.
\end{Le}
Clearly, it suffices to prove this lemma for the functions $\theta_v(x-\tau)$ instead of $\rho_{m,v}(x,\tau)$.
The corresponding proof can be found, for example, in~\cite{Ki}. Also we provide its sketch in the last section.

Let $h = \{h_{m,v}\}$ be an arbitrary $l^2$-va\-lu\-ed 
function such that $M_{i,2,s}h$ is finite at some point. 
We introduce a certain modification of $R$ whose action on any such $h$ is well defined. Namely, we may argue, e.g., as in Section~4.4 of~\cite{KiKr}. Consider some intervals
$J_1 \subset J_2 \subset \cdots$ whose union coincides with $\mathbb{R}$. 
The $L^2$-boundedness of $R$, together with Lemmas~\ref{fIsSuit} and~\ref{Smooth1},
implies that on each interval~$J_\alpha$, we may construct the following $l^2$-valued function:
\begin{multline*}
W_{\alpha,h}(x) = R(\chi_{2J_{\alpha}}h)(x)\\
\hskip25pt+
\bigg\{\int\limits_{\mathbb{R}\setminus 2J_{\alpha}} \sum_{v=-\infty}^{N_m}\big(\rho_{m,v}(x,\tau)-p_{m,v,J_{\alpha}}(x,\tau)\big)h_{m,v}(\tau)\,d\tau\bigg\}_{m\in\mathcal{M}},\\
 x\in J_{\alpha}.
\end{multline*}
Repeating the arguments from Section~4.4 of~\cite{KiKr}, we can prove that each function $W_{\alpha+1,h}(x)-W_{\alpha,h}(x)$, $x\in J_{\alpha}$, 
is equal to some polynomial in $\mathcal{P}_r(l^2)$. Thus, we can define a modification~$\widetilde R$ of the operator~$R$ as follows.
We set $\widetilde Rh = W_{1,h}$ on $J_1$. Next, we set $\widetilde Rh = W_{2,h} + P$ on $J_2$, where $P$ is polynomial in $\mathcal{P}_r(l^2)$ such that 
$P=W_{2,h}-W_{1,h}$ on $J_1$. Further, we extend $\widetilde Rh$ to $J_3$ in the same way, and so on.
\begin{Rem}\label{RPIsP}
For any $l^2$-valued polynomial~$q$, we have $\widetilde R q \in \mathcal{P}_r(l^2)$.
\end{Rem}
\begin{proof}
We can use the arguments similar to those in Section~4.4.1 of~\cite{KiKr}. Namely, we introduce the functions
\begin{multline}\label{PartsOfParts}
W_{\alpha,h}^{\nu}(x) = 
\bigg\{\sum_{v=\nu}^{N_m}\bigg(\int\limits_{\mathbb{R}} \rho_{m,v}(x,\tau)\,h_{m,v}(\tau)\,d\tau\\
-\!\!\!\int\limits_{\mathbb{R}\setminus 2J_{\alpha}}\!\!\! p_{m,v,J_{\alpha}}(x,\tau)\,h_{m,v}(\tau)\,d\tau\bigg)\bigg\}_{m\in\mathcal{M}},\quad
x\in J_{\alpha}.
\end{multline}
It is easy to see that $W_{\alpha,h}^{\nu} \to W_{\alpha,h}$ in $L^2(J_{\alpha},l^2)$ 
as $\nu\to-\infty$. Let $h$ be our polynomial~$q$. Then since $W_{\alpha,q}^{\nu}$ are polynomials (by Fact~\ref{RPIsZero}), the function $W_{\alpha,q}$ is 
a polynomial itself.
\end{proof}

Together with Facts~\ref{RIsL2Bounded} and~\ref{RPIsP}, Lemma~\ref{Smooth1} implies the following statement, which can be proved in the
same way as Theorem~4.21 in Section~4.4.1 of~\cite{KiKr} (in the present paper, we will also prove a statement in the same spirit for a more complicated operator, see Lemma~\ref{HIsBounded}).
\begin{Le}\label{RIsBounded}
	Let $h$ be a measurable $l^2$-valued function such that $M_{i,2,s}h$ is finite at some point.
	Then
	$$
		M_{r,2,s}(\widetilde Rh) \le C M_{i,2,s}h.
	$$
\end{Le}

Suppose 
\begin{equation}\label{h}
	h = \{\Phi_{m,v}\ast g_{m,v}\}_{m\in\mathcal{M},\,v\le N_m}.
\end{equation} 
Lemma~\ref{RIsBounded} allows us to reduce Theorem~\ref{ExMainT} to the estimate $M_{r,2,s}h \le CM_{i,2,s}f$. Indeed, suppose this estimate is fulfilled.
By Fact~\ref{AlmostGood}, we can choose a collection~$P$ of polynomials in $\mathcal{P}_r$ such that
$\tilde h = h - P$ is a function in $L^2_{\mathrm{loc}}(l^2)$. By Lemma~\ref{Smooth1}, we obtain 
$M_{r,2,s}(\widetilde R\tilde h) \le C M_{r,2,s}h$ (here we put $i=r$).
On the other hand, we have 
$$
\widetilde R\tilde h = 
W_{\alpha,\tilde h} + P_{\alpha} = \lim_{\nu\to-\infty} \big(W_{\alpha,\tilde h}^{\nu} + P_{\alpha}^{\nu}\big)\quad\mbox{on each}\quad J_{\alpha},
$$
where the functions $W_{\alpha,\tilde h}^{\nu}$ are defined by~\eqref{PartsOfParts}, $P_{\alpha}$ are polynomials ($l^2$-valued) such that 
the functions
$W_{\alpha,\tilde h} + P_{\alpha}$ agree on different~$J_\alpha$, the polynomials~$P_{\alpha}^{\nu}$ are chosen by the same principle for 
$W_{\alpha,\tilde h}^{\nu}$, and the limits are taken in $L^2(J_{\alpha},l^2)$. 
By identity~\eqref{Decomp} and Fact~\ref{RPIsZero}, we have
$$
W_{\alpha,\,\tilde h}^{\nu}(x) + P_{\alpha}^{\nu}(x) = 
\Big\{e^{-2\pi i\,a_m x}\big(\widehat f\psi^1_{m,\nu}\big)^{\vee}(x) - p^1_{m,\nu}(x)\Big\}_{m\in\mathcal{M}},
$$
where $p^1_{m,\nu}$ are some polynomials that do not depend on~$\alpha$.
So we reduce our theorem to the estimate $M_{r,2,s}h \le CM_{i,2,s}f$, where $h$ is defined by~\eqref{h}.

\section{Conclusion of the proof}\label{Conc}
First, we get rid of the functions~$\Phi_{m,v}$. For this purpose, we
introduce a simple operator~$\boldsymbol\Phi$:
$$
	\boldsymbol\Phi\big(\{g_{m,v}\}\big) = \big\{\Phi_{m,v}\ast g_{m,v}\big\}_{m\in\mathcal{M},\,v\le N_m}.
$$
By Fact~\ref{AlmostGood} and Lemma~\ref{fIsSuit}, the operator~$\boldsymbol\Phi$ is well defined for any sequence of measurable functions $g=\{g_{m,v}\}$, provided $M_{i,2,s}g$ is finite at some point. The Plancherel theorem immediately implies that $\boldsymbol\Phi$ is $L^2$-bounded. 
Also it is clear that if $P$ is a sequence of polynomials in $\mathcal{P}_i$, then each element of $\boldsymbol\Phi P$ is also a polynomial in $\mathcal{P}_i$.
Finally, we may state the following smoothness condition for the kernel of~$\boldsymbol\Phi$.
\begin{Le}\label{Smooth2}
	Fix some $m\in\mathcal{M}$ and $v\in\mathbb{Z}$\textup, $v\le N_m$. Then for any interval~$I$ 
	with center~$\tau_0$\textup, 
	there exist a function $p_{m,v,I}(\tau,t)$ with the following properties.
	First\textup,
	$$
		p_{m,v,I}(\tau,t) = \sum_{\alpha = 0}^{r-1} \tau^{\alpha}\,b_{\alpha}^{m,v,I}(t),
	$$
	where $b_{\alpha}^{m,v,I}$ are functions in $\mathcal{S}$.
	Second\textup,
	$$
		\big|\Phi_{m,v}(\tau-t)-p_{m,v,I}(\tau,t)\big|\le 
		\frac{C_r\, |I|^{r}}{|t-\tau_0|^{r+1}}\quad\mbox{for}\quad t\notin 2I \quad\mbox{and}\quad \tau \in I.
	$$
	Here the constant~$C_r$ does not depend on $m$ or $v$.
\end{Le}
The proof of this lemma is easy. However it will be presented in the last section.
Lemma~\ref{Smooth2}, together with the $L^2$-boundedness of~$\boldsymbol{\Phi}$ and the fact that $\boldsymbol\Phi$ transforms polynomials into 
polynomials, implies the following lemma, which can be proved in the
same way as Theorem~4.21 in~\cite{KiKr} (but without complications concerning the definition of the operator). See also the proof of Lemma~\ref{HIsBounded}, where a more complicated operator is treated.
\begin{Le}\label{PhiIsBounded}
	Let $g$ be a sequence of measurable functions such that $M_{i,2,s}g$ is finite at some point.
	Then
	$$
		M_{r,2,s}(\boldsymbol{\Phi}g) \le C M_{i,2,s}g.
	$$
\end{Le}

So it remains to prove that for the functions~$g_{m,v}$ defined by~\eqref{g}, we have 
\begin{equation}\label{FinalEstimate}
	M_{r,2,s}\big(\{g_{m,v}\}_{m\in\mathcal{M},\,v\le N_m}\big) \le C M_{i,2,s}f.
\end{equation}

\paragraph{Rubio de Francia's operators.}
We split each sequence $\{v\in\mathbb{Z}\colon v\le N_m\}$ into $100$ subsequences~$\mathcal{V}_m^d$ (here $d=1,\dots,100$) in
the following way: $v$ gets into $\mathcal{V}_m^d$ if the residue of $N_m-v$ modulo~$100$ equals~$d$. 
The following remark is almost obvious.
\begin{Rem}\label{IntervalsNotIntersect}
Each collection $\big\{J_{k_{m,v},\,j_{m,v}}\big\}_{m\in\mathbb{\mathcal{M}},\,v\in\mathcal{V}_m^d}$, $d=1,\dots,100$, consists of pairwise
disjoint intervals, provided $A$ is sufficiently close to~$1$. Moreover, these intervals do not contain the origin.
\end{Rem}

Now let $\mathcal{A}$ be any subset of $\mathbb{Z}^2$ such that $\{J_{k,j}\}_{(k,j)\in\mathcal{A}}$ is a collection of pairwise
disjoint intervals that do not contain the origin.
Consider the operator~$H$ defined by the formula
\begin{equation}\label{DefOfH}
	Hf(t)= 
	\int\limits_{\mathbb{R}}
	\kappa(t,y)f(y) \,dy	
\end{equation}
where 
$$
	\kappa(t,y) = \big\{\kappa_{k,j}(t,y)\big\}_{(k,j) \in \mathcal{A}} = 
	\big\{\phi_k(t-y)e^{-2\pi i\,j2^k y}\big\}_{(k,j) \in \mathcal{A}}.
$$
Such operators were first considered by Rubio de Francia in~\cite{Ru}. We prove the following lemma, which, together with 
Fact~\ref{IntervalsNotIntersect}, immediately implies 
estimate~\eqref{FinalEstimate}.
\begin{Le}\label{HIsBounded}
	Let $f$ be a measurable function such that 
	$M_{i,2,s}f$ is finite at some point.
	Then we have
	$$
		M_{r,2,s}(Hf) \le C M_{i,2,s}f.
	$$
\end{Le}

As usual, first we study the behavior of~$H$ in $L^2$ and on polynomials.
\begin{Rem}
 The operator~$H$ is $L^2$-bounded.
\end{Rem}
\begin{proof}
We have
$$
	j2^k+\supp\widehat\phi_{k} \subset J_{k,j},
$$
and the collection $\{J_{k,j}\}_{(k,j)\in\mathcal{A}}$ consists of pairwise disjoint intervals.
But by the Plancherel theorem, all this implies the $L^2$-boundedness of~$H$.
\end{proof}

\begin{Rem}\label{HPIsZero}
We have $Hq \equiv 0$ for any polynomial~$q$. 
\end{Rem}
\begin{proof}
The Fourier transform of the function $e^{-2\pi i\, j2^k y}q(y)$ is supported at 
the single point $-j2^k$. Since $0\notin J_{k,j}$ for $(k,j)\in\mathcal{A}$, it readily follows that 
$-j2^k \notin \supp\widehat\phi_{k}$ and, therefore,
$Hq \equiv 0$.
\end{proof}

The kernel $\kappa(t,y)$ satisfies the following smoothness condition.
\begin{Le}\label{Smooth3}
	For any interval $I \subset \mathbb{R}$\textup, there exists an $l^2$-va\-lu\-ed function 
	$$
		q_I(t,y) = \big\{q_{k,j,I}(t,y)\big\}_{(k,j) \in \mathcal{A}}
	$$
	with the following properties.
	First\textup,
	\begin{equation}\label{qIsPolynomial}
		q_{k,j,I}(t,y) = \sum_{\alpha = 0}^{r-1} t^{\alpha}\,b_{\alpha}^{k,j,I}(y),
	\end{equation}
	where $b_{\alpha}^{k,j,I}$ are functions in $\mathcal{S}$.
	Second\textup, for any $\xi\in l^2$\textup, any $t\in I$\textup, and 
	any $\sigma\in\mathbb{N}$\textup, we have	
	\begin{equation}\label{Smooth3Estimate}
		\bigg(\int\limits_{I_\sigma} \big|\langle \kappa(t,y) - q_I(t,y), \xi\rangle_{l^2}\big|^2 \,dy\bigg)^{1/2} \le 
		C\, 2^{-\sigma B_r}\, |I|^{-1/2}\, |\xi|_{l^2},
	\end{equation}
	where $I_\sigma = 2^{\sigma+1} I \setminus 2^{\sigma} I$ and $B_r = r+1/2$.
\end{Le}
Quite similar estimates can be found in \cite{Ki,KiPa,Ru}. However the exponent~$B_r$ was never calculated precisely. 
Since we need its exact value, we provide the proof of Lemma~\ref{Smooth3} in the last section.

Now we have all the components to prove Lemma~\ref{HIsBounded}. Our proof is similar to the proof of Theorem~4.21 in~\cite{KiKr}, but 
slightly more involved due to the complexity of condition~\eqref{Smooth3Estimate}.

\paragraph{Proof of Lemma~\ref{HIsBounded}.} 
Consider a point~$y_0$ such that $\mu = \widetilde M_{i,2,s}(f)(y_0)$ is a finite number (we recall that $\widetilde M_{i,2,s}$ is
defined by~\eqref{DefOfM2}). Then for any interval~$I$ 
with center~$y_0$, we have
$$
	\frac{1}{|I|^s}\bigg(\frac{1}{|I|}\int\limits_{I}|f-f_I|^2\bigg)^{1/2}\le\mu,
$$
where $f_I = \mathfrak{P}_I f$.

Let~$P_I$ be a sequence of polynomials in~$\mathcal{P}_r$ defined by the formula
\begin{multline}\label{HMinusP}
	P_I(t) = Hf(t) - \\\bigg[H\big(\chi_{2I}(f-f_I)\big)(t) + 
	\!\!\int\limits_{\mathbb{R}\setminus 2I}\!\! \big(\kappa(t,y) - q_I(t,y)\big)\big(f(y)-f_I(y)\big)\,dy\bigg].
\end{multline}
The fact that each element of~$P_I$ is well defined and belongs to~$\mathcal{P}_r$ follows immediately from
Lemma~\ref{fIsSuit}, Fact~\ref{HPIsZero}, and identity~\eqref{qIsPolynomial}. Note that we do not require $P_I$ to be an $l^2$-valued polynomial in $\mathcal{P}_r(l^2)$.

If we verify the estimate
$$
	\frac{1}{|I|^s}\bigg(\frac{1}{|I|}\int\limits_{I}\big|Hf - P_I\big|_{l^2}^2\bigg)^{1/2}\le C\mu,
$$ 
then we will prove Lemma~\ref{HIsBounded}.

First, we present some auxiliary propositions.
Applying the Cau\-chy--Schwarz inequality to Lemma~0.24 in~\cite{KiKr}, we obtain the following statement (it is worth noting that the lemma just mentioned is, in its turn, a simple consequence of considerations in~\cite{DeSh}).
\begin{Le}\label{LeFromBook}
For any two cubes $Q\subset Q_1 \subset\mathbb{R}^n$ and any polynomial $p \in \mathcal{P}_{i}$\textup, we have
$$
\|p\|_{L^{\infty}(Q_1)}\le C \bigg(\frac{|Q_1|}{|Q|}\bigg)^{\frac{i}{n}}\bigg(\frac{1}{|Q|}\int\limits_{Q}|p|^2\bigg)^{1/2},
$$
where a constant~$C$ depends only on $i$ and $n$.
\end{Le}
The properties of~$\mathfrak{P}_I$ imply the following estimate (which also can be found in~\cite{KiKr}).
\begin{Rem}\label{RemFromBook}
For any interval~$I$, we have
$$
\bigg(\frac{1}{|I|}\int\limits_I |f_{2I}-f_I|^2\bigg)^{1/2} \le C\bigg(\frac{1}{|2I|}\int\limits_{2I}|f-f_{2I}|^2\bigg)^{1/2}.
$$
\end{Rem}
\begin{proof}
\begin{multline*}
	\bigg(\frac{1}{|I|}\int\limits_I |f_{2I}-f_I|^2\bigg)^{1/2} = 
	\bigg(\frac{1}{|I|}\int\limits_I |\mathfrak{P}_I(f-f_{2I})|^2\bigg)^{1/2}\\ \le
	C\bigg(\frac{1}{|I|}\int\limits_I |f-f_{2I}|^2\bigg)^{1/2} \le	
	C'\bigg(\frac{1}{|2I|}\int\limits_{2I}|f-f_{2I}|^2\bigg)^{1/2}.
\end{multline*}
\end{proof}

Now we consider the first summand in square brackets in~\eqref{HMinusP}.
Using the $L^2$-boundedness of~$H$, we obtain
\begin{align*}
	\bigg(\frac{1}{|I|}\int\limits_I \big|H\big(\chi_{2I}(f-&f_I)\big)\big|_{l^2}^2\bigg)^{1/2} \\
	&\le C\bigg(\frac{1}{|I|} \int\limits_{2I}|f-f_I|^2\bigg)^{1/2}\\
	&\le C\bigg(\frac{1}{|I|}\int\limits_{2I}|f-f_{2I}|^2\bigg)^{1/2} + C'\,\|f_{2I}-f_I\|_{L^{\infty}(2I)}.
\end{align*}
By Lemma~\ref{LeFromBook} and Fact~\ref{RemFromBook}, we have
\begin{align*}
	\|f_{2I}-f_I\|_{L^{\infty}(2I)} &\le C\, 2^i\,\bigg(\frac{1}{|I|}\int\limits_{I}|f_{2I}-f_I|^2\bigg)^{1/2}\\
	&\le C'\bigg(\frac{1}{|2I|}\int\limits_{2I}|f-f_{2I}|^2\bigg)^{1/2}.
\end{align*}
Combining this inequality with the previous one, we arrive at the following estimate:
$$
\frac{1}{|I|^s}\bigg(\frac{1}{|I|}\int\limits_I \big|H\big(\chi_{2I}(f-f_I)\big)\big|_{l^2}^2\bigg)^{1/2} \le C\mu.
$$

It remains to treat the second summand in square brackets in~\eqref{HMinusP}. We set
$$
	W(t) = \!\!\int\limits_{\mathbb{R}\setminus 2I}\!\! \big(\kappa(t,y) - q_I(t,y)\big)\big(f(y)-f_I(y)\big)\,dy.
$$
We can write
\begin{align*}
	&\hskip-10pt\bigg(\frac{1}{|I|}\int\limits_{I}\big|W\big|_{l^2}^2\bigg)^{1/2}\\
	\nonumber &=|I|^{-1/2}\!\!\!\!\sup_{u:\:\|u\|\le 1}\bigg|\int\limits_{I}\big\langle u(t), 
	W(t)\big\rangle_{l^2} \,dt\bigg|\\
	\nonumber &\le |I|^{-1/2}\!\!\!\!\sup_{u:\:\|u\|\le 1} \int\limits_{I}\int\limits_{\mathbb{R}\setminus 2I}\!\! 
	\big|f(y)-f_I(y)\big|\,\big|\big\langle u(t), \kappa(t,y)-q_I(t,y)\big\rangle_{l^2}\big| \,dy\, dt,	
\end{align*}
where the supremum is taken over the unit ball in $L^2(I,l^2)$. By the Cau\-chy--Schwarz inequality, we see that
the last expression is not greater than
\begin{align*}
	|I|^{-1/2}\!\!\!\!\sup_{u:\:\|u\|\le 1}\sum_{\sigma=1}^\infty\, &\bigg(\int\limits_{I_\sigma} \big|f(y)-f_I(y)\big|^2 \,dy\bigg)^{1/2}\\
	&\times\int\limits_{I} \bigg(\int\limits_{I_\sigma}\big|\big\langle u(t), \kappa(t,y)-q_I(t,y)\big\rangle_{l^2}\big|^2 \,dy\bigg)^{1/2} dt.
\end{align*}
Lemma~\ref{Smooth3} implies that the second factor in each summand can be estimated~by
$$
C\, 2^{-\sigma (1/2+r)}\, |I|^{-1/2}\int\limits_{I} |u(t)|_{l^2}\,dt \le C'\,2^{-\sigma (1/2+r)}.
$$
Therefore, we must estimate the following expression:
\begin{equation*}
|I|^{-1/2}\sum_{\sigma=1}^\infty 2^{-\sigma (1/2+r)} \bigg(\int\limits_{I_\sigma} |f-f_I|^2\bigg)^{1/2}.
\end{equation*}
We can write
\begin{align*}
\bigg(\int\limits_{2^{\sigma+1}I}\!\! &|f-f_I|^2\;\bigg)^{1/2}\\
&\le \bigg(\int\limits_{2^{\sigma+1}I}\!\! |f-f_{2^{\sigma+1}I}|^2\;\bigg)^{1/2} + 
\sum_{\alpha=1}^{\sigma}|2^{\sigma+1}I|^{1/2}\,\|f_{2^{\alpha+1}I}-f_{2^{\alpha}I}\|_{L^{\infty}(2^{\sigma+1}I)}.
\end{align*}
Applying Lemma~\ref{LeFromBook} and Fact~\ref{RemFromBook}, we obtain
\begin{align*}
	\|f_{2^{\alpha+1}I}-f_{2^{\alpha}I}\|_{L^{\infty}(2^{\sigma+1}I)}&\le 
	C\bigg(\frac{|2^{\sigma+1}I|}{|2^{\alpha}I|}\bigg)^{i-1}
	\bigg(\frac{1}{|2^{\alpha}I|}\int\limits_{2^{\alpha}I}|f_{2^{\alpha+1}I}-f_{2^{\alpha}I}|^2\bigg)^{1/2}\\
	&\le C'\, 2^{(\sigma-\alpha)(i-1)}\bigg(\frac{1}{|2^{\alpha+1}I|}\int\limits_{2^{\alpha+1}I}|f-f_{2^{\alpha+1}I}|^2\bigg)^{1/2}\\ 
	&\le C'\, 2^{(\sigma-\alpha)(i-1)}\, |2^{\alpha+1}I|^s\, \mu.
\end{align*}
Substituting the last expression into the previous estimate, we get
$$
	\bigg(\int\limits_{I_{\sigma}} |f-f_I|^2\;\bigg)^{1/2}\le 
	C\, |2^{\sigma+1}I|^{1/2}\,|I|^s\,\mu\,2^{\sigma(i-1)}\sum_{\alpha=1}^{\sigma} 2^{\alpha(s-i+1)}.
$$
Finally, we have
$$
	\frac{1}{|I|^s}\bigg(\frac{1}{|I|}\int\limits_I \big|W\big|_{l^2}^2\bigg)^{1/2} \le 
	C\mu\sum_{\sigma=1}^{\infty}\Big( 2^{-\sigma(r-i+1)}\sum_{\alpha=1}^{\sigma}2^{\alpha(s-i+1)}\Big).
$$
The sum over~$\alpha$ is dominated by a constant if $s-i+1 < 0$, by~$\sigma$ if $s-i+1 = 0$, and by $C2^{\sigma(s-i+1)}$ if $s-i+1 > 0$.
Since $r>\max \{s,i-1\}$, we see that the series in~$\sigma$ is convergent in any case, and we are done.

\section{Smoothness conditions}
In this chapter we prove Lemmas~\ref{Smooth1}, \ref{Smooth2}, and~\ref{Smooth3}.

\paragraph{Proof of Lemma~\ref{Smooth1}.}
Such estimates are widely known, and the corresponding proof can be found, for example, in~\cite{Ki}. Here we present its sketch.

Let $P_{u,v}$ be the Taylor polynomial of $\theta_v$ at the point $u$ of degree $r-1$.
Put $p_{v,I}(x,\tau) = P_{x_0-\tau,v}(x-\tau)$. 
Since $\theta \in \mathcal{S}$,
it is easily seen that
\begin{multline*}
	|\theta_v(x-\tau)-p_{v,I}(x,\tau)| \le C_{r,\beta}\, A^{v(r+1)}\,|x-x_0|^{r}\,(1 + A^v |\tau - x_0|)^{-\beta},\\
	\beta = 1,2,\dots.
\end{multline*}
Letting $\beta=0$ and $\beta=r+2$, we obtain
\begin{align}
	&|\theta_v(x-\tau)-p_{v,I}(x,\tau)| \le C_r\,A^{v(r+1)}\,|I|^{r},\label{est1}\\
	&|\theta_v(x-\tau)-p_{v,I}(x,\tau)| \le C_r\,A^{-v}\frac{|I|^{r}}{|\tau-x_0|^{r+2}}.\label{est2}
\end{align}
Estimate (\ref{est1}) is better then \eqref{est2} exactly when $A^v < \frac{1}{|\tau-x_0|}$. 
We set $p_{m,v,I}(x,\tau) = e^{2\pi i\,\delta_{m,v} \tau}p_{v,I}(x,\tau)$ and
split the sum in~\eqref{Smooth1Estimate} into two: the first over all $v$ such that $A^v < \frac{1}{|\tau-x_0|}$ and 
the second over all remaining~$v$. Using~\eqref{est1} for the summands in the first part and~\eqref{est2} for the rest, 
we obtain~\eqref{Smooth1Estimate}.

\paragraph{Proof of Lemma~\ref{Smooth2}.}
Let $P_{u,m,v}$ be the Taylor polynomial of $\Phi_{m,v}$ at the point $u$ of degree $r-1$. We recall that 
$
	\Phi_{m,v}(t)=e^{-2\pi i\,\delta_{m,v}t}\varphi_m(t)
$
and $\varphi_m(t) = l_m \varphi(l_m t)$, where $\varphi$ is a certain function in~$\mathcal{S}$.
Put $p_{m,v,I}(\tau,t) = P_{\tau_0-t,m,v}(\tau-t)$. Then we have
$$
|\Phi_{m,v}(\tau-t)-p_{m,v,I}(\tau,t)|\le 
C_r\!\!\sum_{\substack{0\le \alpha,\alpha' \le r \\ \alpha + \alpha' = r}}\!\!|\tau-\tau_0|^r\,\delta_{m,v}^{\alpha} |\varphi_m^{(\alpha')}(\eta)|,
$$
where $\eta$ lies between $\tau-t$ and $\tau_0-t$. Since $|\tau-t|\asymp|\tau_0-t|$,
we obtain
\begin{multline*}
|\tau-\tau_0|^r\,\delta_{m,v}^{\alpha} |\varphi_m^{(\alpha')}(\eta)| \le 
C_{r,\beta}\,|\tau-\tau_0|^r\,\delta_{m,v}^{\alpha} \,l_m^{\alpha'+1}(1+l_m|t-\tau_0|)^{-\beta},\\
	\beta = 1,2,\dots.
\end{multline*}
Note that $\delta_{m,v}\le l_m$. Thus, letting $\beta=0$ and $\beta=r+2$, we have
\begin{align*}
	&|\Phi_{m,v}(\tau-t)-p_{m,v,I}(\tau,t)| \le C_r\,l_m^{r+1}\,|I|^{r},\\
	&|\Phi_{m,v}(\tau-t)-p_{m,v,I}(\tau,t)| \le \frac{C_r|I|^{r}}{l_m|t-\tau_0|^{r+2}}.
\end{align*}
Using the first inequality if $l_m < \frac{1}{|t-\tau_0|}$ and the second otherwise, we arrive at the desired estimate.

\paragraph{Proof of Lemma~\ref{Smooth3}.}
Similar estimates and their proofs can be found in \mbox{\cite{Ki,KiPa,Ru}}. For completeness, we repeat the proof here. At the same time, we 
calculate~$B_r$.

We take the whole~$\mathbb{Z}^2$ instead of~$\mathcal{A}$. If we prove the lemma in such a setting, then we will be able to get the required estimate for any $\mathcal{A} \subset \mathbb{Z}^2$:
we will only need to consider the vectors $\xi = \{\xi_{k,j}\}_{(k,j)\in\mathbb{Z}^2}$ such that $\xi_{k,j} = 0$ 
for $(k,j) \notin \mathcal{A}$.

We define the polynomials $p_{k,I}(t,y)$ for the functions $\phi_k$ in the same way as we defined $p_{v,I}$ 
for~$\theta_v$.
Estimates~\eqref{est1} and~\eqref{est2} are fulfilled for~$\phi_k$ and~$p_{k,I}$ when 
$A = 2$. We define the function $q_I(t,y)$ by the formula
$$
	q_I(t,y) = \big\{q_{k,j,I}(t,y)\big\}_{(k,j) \in \mathcal{A}} =  \big\{p_{k,I}(t,y)e^{-2\pi i\,j2^k y}\big\}_{(k,j) \in \mathcal{A}}.
$$
Also we put
\begin{displaymath}
	\gamma_{k,\sigma} = \sup_{y \in I_\sigma,\, t \in I}|\phi_k(t-y) - p_{k,I}(t,y)|.
\end{displaymath}
In this notation, we have
\begin{equation}\label{Smooth3Decomp}
\begin{aligned}
	\bigg(\int\limits_{I_\sigma} &\big|\langle\kappa(t,y) - q_I(t,y), \xi\rangle\big|^2 \,dy\bigg)^{1/2}\\
	&\le\bigg(\int\limits_{I_\sigma}\Big(\sum_{k \in \mathbb{Z}}|\phi_k(t-y) - p_{k,I}(t,y)|
	\Big|\sum_{j \in \mathbb{Z}} \xi_{k,j}\,e^{-2\pi i\,j2^k y} \Big|\Big)^2 \,dy\bigg)^{1/2}\\
	&\le\sum_k \gamma_{k,\sigma} \bigg(\int\limits_{I_\sigma} \Big|\sum_j \xi_{k,j}\,e^{-2\pi i\,j2^k y} \Big|^2\, dy\bigg)^{1/2}\\
	&\le\Big(\sum_k \gamma_{k,\sigma}\Big)^{1/2}
	\bigg(\sum_k \gamma_{k,\sigma} \int\limits_{I_\sigma}\Big|\sum_j \xi_{k,j}\,e^{-2\pi i\,j2^k y}\Big|^2\,dy\bigg)^{1/2};
\end{aligned}
\end{equation}
here we have used the triangle inequality in $L^2$ and the Cauchy inequality for sums. 

Next, by~\eqref{est1} and~\eqref{est2},
we obtain
\begin{equation}\label{EstimatesOnGamma}
	\gamma_{k,\sigma} \le C_r\, 2^{k(r+1)}\, |I|^{r},\quad \gamma_{k,\sigma} \le C_r\, 2^{-k}\, |I|^{-2}\, 2^{-(r+2)\sigma}.
\end{equation}
The second estimate is stronger then the first exactly when $2^k \ge |I|^{-1}2^{-\sigma}$. Splitting the sum $\sum\gamma_{k,\sigma}$ into
two parts accordingly, we have
\begin{equation}\label{FirstFactor}
\begin{aligned}
	\sum_{k} \gamma_{k,\sigma} &\le
	C_r \Big(\sum_{k:\:2^k < |I|^{-1}2^{-\sigma}}\!\!\!\!\!\!\!\! 2^{k(r+1)}\, |I|^{r} +
	\!\!\!\!\sum_{k:\:2^k\ge|I|^{-1}2^{-\sigma}}\!\!\!\!\!\!\!\! 2^{-k}\, |I|^{-2}\, 2^{-(r+2)\sigma}\Big)\\
	&\le C_r'\, 2^{-\sigma(r+1)}\, |I|^{-1}.
\rule{0pt}{20pt}
\end{aligned}
\end{equation}

It remains to estimate the second factor in the last expression in \eqref{Smooth3Decomp}. Making the substitution $\tilde y = 2^ky$ and using
the Riesz--Fischer theorem, we have
$$
	\int\limits_{2^{\sigma+1}I}\!\!\Big|\sum_{j}\xi_{k,j}\,e^{-2\pi i\, j2^k y}\Big|^2\, dy \le  
	\begin{cases}
		 \displaystyle C\,2^{\sigma+1}\,|I|\sum_{j}|\xi_{k,j}|^2, &2^{\sigma+1}|I|\ge 2^{-k};\\
		 \displaystyle 2^{-k}\sum_{j}|\xi_{k,j}|^2, &2^{\sigma+1}|I| < 2^{-k}.
	\end{cases}
$$
By these estimates and inequalities \eqref{EstimatesOnGamma}, we obtain
\begin{align*}
	\sum_k \gamma_{k,\sigma} &\int\limits_{I_\sigma}\Big|\sum_j \xi_{k,j}\,e^{-2\pi i\,j2^k y}\Big|^2\,dy\\
	&\le C\,|\xi|_{l^2}^2\Big(\sum_{k:\:2^k < |I|^{-1}2^{-\sigma-1}}\!\!\!\!\!\!\!\!\!\! \gamma_{k,\sigma}\, 2^{-k} +\!\!\!\!\! 
	\sum_{k:\:2^k\ge|I|^{-1}2^{-\sigma-1}}\!\!\!\!\!\!\!\!\!\! \gamma_{k,\sigma}\,2^{\sigma+1}\,|I|\Big) \\
	&\le C_r\,|\xi|_{l^2}^2\Big(\sum_{k:\:2^k < |I|^{-1}2^{-\sigma-1}}\!\!\!\!\!\!\!\!\!\! 2^{kr}\,|I|^{r}+\!\!\!\!\!
	\sum_{k:\:2^k\ge|I|^{-1}2^{-\sigma-1}}\!\!\!\!\!\!\!\!\!\! 2^{-k}\,2^{-(r+1)\sigma+1}\,|I|^{-1}\Big)
\rule{0pt}{20pt}\\
	&\le C_r'\,|\xi|_{l^2}^2\,2^{-r\sigma}.
\rule{0pt}{20pt}
\end{align*}
Combining this estimate with~\eqref{FirstFactor}, we conclude the proof.

\section{Acknowledgments}
The author expresses his gratitude to S.~V.~Kislyakov for valuable assistance in obtaining a proper formulation of the problem as well as for his helpful comments that have significantly improved the text.

\end{document}